\documentclass[11pt,leqno]{amsart}
\usepackage{amsmath, amsthm, amssymb,xspace}
\usepackage[breaklinks=true]{hyperref}
\usepackage{amsmath,amscd}
\theoremstyle{plain}
\newtheorem{theorem}{Theorem}[section]
\newtheorem{lemma}{Lemma}[section]
\newtheorem{prop}{Proposition}[section]
\newtheorem{cor}{Corollary}[section]

\theoremstyle{definition}

\newcommand{\bggo}{\mathcal O}

\newcommand{\mf}[1]{\displaystyle{\mathfrak{#1}}}

\newcommand{\comment}[1]{}

\DeclareMathOperator{\spec}{\ensuremath{Spec}}
\DeclareMathOperator{\Supp}{\ensuremath{Supp}}
\DeclareMathOperator{\Gr}{\ensuremath{gr}}

\DeclareMathOperator{\ad}{\ensuremath{ad}}

\DeclareMathOperator{\Sym}{\ensuremath{Sym}}
\DeclareMathOperator{\Ann}{\ensuremath{Ann}}

\begin{document}
\title[completions of Infinitesimal Hecke Algebras ]{Completions of infinitesimal Hecke algebras of $\mf{sl}_2$}

\author{Akaki Tikaradze}
\address{The University of Toledo, Department of Mathematics, Toledo, Ohio, USA}
\email{\tt tikar06@gmail.com}
\begin{abstract}

We relate completions of infinitesimal Hecke algebras of $\mf{sl}_2$
to noncommutative deformations of Kleinian singularities of type $D$ of
Crawley-Boevey and Holland. As a consequence, we show an analogue of
Bernstein's inequality and simplicity of generic maximal primitive
quotients of these algebras. We also establish Skryabin type 
equivalence for these algebras.
\end{abstract}
\maketitle

\section{Introduction}

Given an associative Noetherian  $\mathbb{C}$-algebra $A$ of finite Gelfand-Kirillov
dimension, it is natural to ask
if (generalized) Bernstein's inequality holds: Is it true that
for any finitely generated $A$-module $M$ one has $GK_A(M)\geq \frac{1}{2}GK(A/\Ann(M))?$
(here $GK$ stands for Gelfand-Kirillov dimension). 
We show that this is the case for infinitesimal Hecke algebras
of $\mf{sl}_2.$

Recall that for a reductive Lie algebra $\mf{g}$ and its finite
dimensional representation $V$, Etingof, Gan and Ginzburg \cite{EGG}
defined a family of PBW deformations of the universal enveloping algebra
$\mf{U}(\mf{g} \ltimes V)$ of the semi-direct product Lie algebra
$\mf{g}\ltimes V$, which they called infinitesimal Hecke algebras.

There is a particularly nice family of infinitesimal Hecke algebras for
the pair $\mf{g}=\mf{sl}_2, V=\mathbb{C}^2$. These algebras, denoted by $H_z,$
 depend on a parameter $z$ which is an element of the center of $\mf{U}\mf{sl}_2$. 
 Algebras $H_z$
 were studied by Khare and the author \cite{Kh,KT}. In this
paper, we relate noncommutative deformations of Kleinian singularities of
type $D$ of Crawley-Boevey and Holland (which are spherical subalgebras
of symplectic reflection algebras for $\dim V=2$) \cite{CBH} to
infinitesimal Hecke algebras of $\mf{sl}_2$. Namely, we show that a certain
completion of the central quotient of an infinitesimal Hecke algebra of
$\mf{sl}_2$ is isomorphic to the tensor product of the completed Weyl
algebra in two generators and the completion of an algebra of
Crawley-Boevey and Holland. As a consequence, we establish Bernstein's inequality for
these algebras  (Theorem \ref{Bern}.) We also show an equivalence
between certain subcategory of modules over $H_z$ and the category of modules
over corresponding noncommutative deformations of type $D$ singularities(Theorem\ref{skryabin}).

\section{Completions of almost commutative algebras}
We will always work over the field of complex numbers $\mathbb{C}.$
In this section we will recall some necessary constructions
and fix notations related to a slice algebra construction due to Losev\cite{L1}.

Throughout we will use the following convention: For a commutative algebra
$B$ and a closed point $y\in \spec B, B_{\bar{y}}$ will denote the completion of $B$
with respect to the maximal ideal $m_y$ corresponding to $y.$ 
Also, for a symplecic variety $Y$ and $y\in Y$ we will denote
by $W_{\hbar}(Y_{\bar{y}})$ the completed Weyl algebra, a deformation
 quantization of $\bggo{}(Y)_{\bar{y}},$ where $\hbar$ is the deformation
 parameter.

By an almost commutative algebra we will mean an
associative $\mathbb{C}$-algebra equipped with an ascending filtration
$$\mathbb{C}=A_0 \subset A_1 \subset \cdots A_n \subset \cdots, \qquad
\bigcup_{n \in \mathbb{N}} A_n = A, \qquad A_n A_m \subset A_{n+m},$$

\noindent such that the associated graded algebra is a finitely generated
commutative ring over $\mathbb{C}$. Recall that in this case $\Gr A$
comes equipped with a natural Poisson bracket. The origin of $\spec \Gr A$ will
be denoted by $\{ 0 \}$. 

By a generic subset of an algebraic variety we
will mean Weil generic subset, i.e. a set which is the complement of a countable union of proper
closed subvarieties.

Let $A$ be an almost commutative algebra equipped with a filtration $A_n$
($n\geq 0$) over $\mathbb{C}$.
Let $Y\subset \spec \Gr A=X$ be an algebraic symplectic leaf $\dim Y=2d$,
and let $y\in Y$. We will recall the slice algebra construction of Losev
\cite{L1}.
One starts by completing the Rees algebra $R(A)=\bigoplus_n A_n \hbar^n
\subset A[\hbar]$ with respect to the ideal $p^{-1}(m_y)$, where
$m_y\subset \Gr A$ is the maximal ideal corresponding to $y$, and
$p : R(A) \to R(A) / \hbar = \Gr A$ is the natural projection.
This completed algebra will be denoted by $R(A)_{\bar{y}}$. Then
according to Losev, $R(A)_{\bar{y}}$ is a free
$\mathbb{C}[[\hbar]]$-module and $R(A)_{\bar{y}} /\hbar = (\Gr
A)_{\bar{y}}$, the completion of $\Gr A$ at $m_y$.

According to Kaledin \cite{Ka}, 
$(\Gr A)_{\bar{y}}$ is isomorphic to the completed tensor product 
$\spec {Y}_{\bar{y}}\otimes B$, where $B$ is a complete 
Poisson algebra with the
origin being a symplectic leaf of $\spec B$. As proved by Losev
\cite{L1}, this decomposition can be lifted to $R(A)_{\bar{y}},$ meaning
that there is a free $\mathbb{C}[[\hbar]]$-subalgebra
${A}^{\spadesuit}_y$, a {\it slice algebra}, such that $R(A)_{\bar{y}}$
is identified with the completed tensor product
$W_{\hbar}(Y_{\bar{y}})\otimes_{\mathbb{C}[[\hbar]]} A^{\spadesuit}_y$,
such that   ${A}^{\spadesuit}_y/\hbar = B$. 

If $M$ is a finitely generated left $A$-module, then $M$ can be equipped
with a filtration compatible with the filtration of $A$ such that the
corresponding Rees module $R(M)$ is finitely generated over $R(A)$ (a
good filtration). As before, one defines $R(M)_{\bar{y}}$ as the
completion of $R(M)$ with respect to $p^{-1}(m_y),$ 
so $R(M)_{\bar{y}}=R(M)\otimes_{R(A)}R(A)_{\bar{y}}$
which is a nonzero
finitely generated $R(A)_{\bar{y}}$-module when $y\in \Supp(\Gr M)$.
In which case $R(M)_{\bar{y}}/\hbar=(\Gr M)\otimes_{\Gr A}(\Gr A)_{\bar{y}}=(\Gr M)_{\bar{y}}.$
We will denote by $SS(M)$ the corresponding characteristic variety $Supp(\Gr M)\subset \spec \Gr A.$
Recall the following observation of Losev.

\begin{prop} \label{simple}
Let $I \subset A$ be a two-sided ideal of $A$ such that $\bar{Y}$
is a connected component of  $V(\Gr I),$ where as before $Y$ is a 
symplectic leaf in $\spec \Gr A$. Then for any $y \in Y,$
there exists a nonzero left $A^{\spadesuit}_y$-module which is a 
finitely generated free module
over $\mathbb{C}[[\hbar]].$ 
\end{prop}

We recall the proof for the convenience of the reader.

\begin{proof}
We have that $\bar{I} = R(I)R(A)_{\bar{y}}$ is a two-sided ideal of
$R(A)_{\bar{y}}$, such that $R(A)_{\bar{y}}/\bar{I}$ has support
$Y_{\bar{y}}$.
Therefore, any finitely generated $A^{\spadesuit}_y$-submodule of
$R(A)_{\bar{y}}/\bar{I}$ is supported at the origin in $\spec
A^{\spadesuit}_y/\hbar,$ hence is finite-dimensional.
\end{proof}

For a finitely generated module $M$ over an almost commutative algebra$A$, we will denote its
Gelfand-Kirillov dimension by $GK(M)$, as usual. Recall that in the above
setting $GK(M)=\dim SS(M).$

\section{Completions of infinitesimal Hecke algebras of $\mf{sl}_2$}

We will denote by $\Delta$ the rescaled Casimir element $\Delta = h^2 +
4fe + 2h \in \mf{U}\mf{sl}_2$, where $h, e, f$ denote the standard basis
elements for $\mf{sl}_2$.
For a given $z'\in \mathbb{C}[\Delta]$, the algebra $H_{z'}$ is the
quotient $\mf{U}\mf{g}\ltimes TV/([x, y]-z'),$ where $\mf{g}=\mf{sl}_2$
and $V=\mathbb{C}x\oplus \mathbb{C}y$ is its standard 2-dimensional
representation with relations
$$[e, x]=0, \qquad [f, x]=y, \qquad [h,x]=x.$$

It was shown in \cite{KT} that center of $H_{z'}$ is generated by $t_z =
ey^2 + \frac{1}{2} h(xy+yx)-fx^2 + z$, where $z \in \mathbb{C}[\Delta]$
is an element uniquely defined by $z'$ up to adding a constant (but $z'$ is
uniquely determined by $z$). One has
$\deg(z)=\deg(z')+1$, and the leading coefficient of $z$ as a polynomial
in $\Delta$ is $\frac{1}{2\deg(z')+1}$ times the leading coefficient of
$z'$.
Let $U_{z}$ be the quotient $H_{z'}/(t_z).$ 
We will introduce the following filtration on $U_{z}$: 
$$\deg e=\deg f=\deg h=2, \qquad \deg x=\deg y=2n+1,$$

\noindent where $n=\deg z'$ as a polynomial in $\Delta$.
From now on we will assume that $n \geq 2$ and $z$ is monic in $\Delta$.
In what follows for a filtered algebra $A,$ given an element
$a\in A,$ we will denote for simplicity $\Gr a\in \Gr A$ still by $a$
whenever it will not cause a confusion.
It follows easily from PBW property that $\Gr U_{z}$ is the quotient of the polynomial algebra
$\mathbb{C}[e, f, h, x, y]$ by the principal ideal
$(ey^2+hxy-fx^2-(\Delta)^{n+1})$, where $\Delta=h^2+4ef$. 
We will denote this Poisson algebra by $B_n$. Thus, the Poisson bracket on $B_n$
is defined as follows
$$\lbrace e, x\rbrace=0=\lbrace f, y\rbrace, \lbrace f, x\rbrace=y, \lbrace x, y\rbrace=(2n+1){\Delta}^n,$$
$$\lbrace h, e\rbrace=2e, \lbrace h, f\rbrace=-2f, \lbrace e, f\rbrace=f, \lbrace e, y\rbrace=x.$$

We note that $SL_2(\mathbb{C})$ acts naturally on $H_{z'}, U_z$ preserving the corresponding filtration.
This action gives rise to a natural action of $SL_2(\mathbb{C})$  on $B_n$ preserving the Poisson
bracket.

We now describe the symplectic leaves of $\spec B_n$.

\begin{prop}
$B_n$ is a normal integral domain. The singular locus of $\spec B_n$ is
$\lbrace x=y=0\rbrace $, and its smooth locus is symplectic. The symplectic leaves
of $\spec B_n$ are the origin $\{0\}, V(I) \setminus \{ 0 \}$, and its smooth locus
$\spec B_n\setminus \lbrace x=y=0\rbrace.$ 
\footnote{Apoorva Khare has obtained this answer earlier using a
different computation.}
\end{prop}

\begin{proof}
It is easy to see that $B_n$ is a domain.
It is clear that the ideal $I=(x, y, \Delta)$ is a Poisson ideal and
$V(I)$ (the zero locus of $I$) belongs to the singular locus of $\spec B_n$. 
Let us
take $a \in \spec B_n\setminus V(I).$ 
 Let $m_a$ be the
corresponding maximal ideal.
The $m_a$-adic completion of $B_n$ will be denoted
by $B_{\bar{a}}$ for brevity. We will show that $\spec B_{\bar{a}}$  a (formal)
symplectic
variety. Using the action of $SL_2(\mathbb{C})$
on $\spec B_n$ we may assume without loss of generality that $e(a)\neq 0.$ 
It is easily seen that $B_{\bar{a}}$ is isomorphic to the tensor product
of $\mathbb{C}[[e^{-1}h-\frac{h(a)}{e(a)}, e-e(a)]]$ and of the
completion of $\mathbb{C}[A, C, \Delta] / (C^2 - \Delta
(\frac{1}{4}A^2+\Delta^{n}))$ at the point $(A(a), C(a), \Delta(a))$,
where $A = e^{-\frac{1}{2}} x, C = e^{\frac{1}{2}} y + \frac{1}{2} h A$.
The latter is the ring of functions of the Kleinian singularity of
type $D_{n+2},$ which is symplectic outside the origin.
Thus, $\spec B_n\setminus V(I)$
is a symplectic. On the other hand, $V(I)$ considered as a reduced
subvariety of $\spec B_n$ is the nilpotent cone of $\mf{sl}_2.$
Hence $V(I)\setminus\lbrace 0\rbrace$ is a simplectic leaf of $\spec B_n$.
Normality of $B_n$ follows Serre's criterion, since $B_n$ is regular in codimension
1.
\end{proof}

The following is an analogue of Kostant's theorem.

\begin{prop}
The algebra $H_{z}$ is a free module over its center.
\end{prop}

\begin{proof}
It suffices to check that $\Gr H_{z}$ is a free module over
$\mathbb{C}[\Gr t_z]$.
Let us introduce another filtration of $\Gr H_z$ by setting $\deg x =
\deg y=0$, $\deg e = \deg f = \deg h = 1$. Then under the new filtration,
$\Gr(\Gr t_z) = \Delta^{n+1}$.
But by Kostant's theorem, $\Sym \mf{g}$ is free over $\mathbb{C}[\Delta^{n+1}]$. 
Thus we conclude that $\Gr (\Gr (H_z))=\mathbb{C}[x, y]\otimes \Sym \mf{g}$ is
free over $\Gr (\Gr t_z)$. Hence $\Gr H_{z}$ is free over 
$\mathbb{C}[\Gr t_z]$.
\end{proof}

Recall that for a given  finite group $\Gamma\in SL_2(\mathbb{C}),$ and a
central element of the group algebra $\lambda\in \mathbb{C}[\Gamma]$,
Crawley-Boevey and Holland \cite{CBH} defined an algebra
$\bggo{}^{\lambda}$ as $e(T(V)/([x, y]-\lambda))e$, 
where $e=\frac{1}{|\Gamma|}\sum_{g \in \Gamma}g$
and $x, y$ is the standard basis of $V=\mathbb{C}^2$.

In what follows we will consider square roots of certain non-central
elements. Let us clarify this and fix the appropriate notation.
Let $A$ be an associative $\mathbb{C}$-domain, and let $e\in A$ be an element
such that $\ad(e)$ acts locally nilpotently on $A.$ 
Moreover, assume that $A$ is separated in $(e-1)$-adic topology, i.e. $\cap_n (e-1)^n=0.$
Then we will denote by $A[e^{-\frac{1}{2}}]$ the subalgebra of the completion of $A$
by $(e-1)$ generated over $A$ by $e^{-\frac{1}{2}}=\exp(-\frac{1}{2}\log e),$
where $\log e$ is understood as a power series in $(e-1).$

Also, by $W_1(R)$ we will denote
the Weyl algebra over $R$ with 2 generators: $W_1(R)=R\langle x, y\rangle/([x, y]-1).$

The following result was motivated by Losev's work on completions
of symplectic reflection algebras \cite{L2}.

\begin{theorem}\label{completion}
For any $a\in Y=V(I) \setminus \{ 0 \},$
the algebra $R(U_z)_{\bar{a}}$ is isomorphic to the completed tensor
product $W_{\hbar}(Y_{\bar{a}})\otimes_{\mathbb{C}[[\hbar]]} R(\bggo{}_{\lambda(z)})_{\bar{0}},$ 
where $\bggo{}^{\lambda(z)}$ is the
noncommutative deformation of the Kleinian singularity of type $D_{n+2}$ 
(parameter $\lambda(z)$ will be determined in the proof). For generic $z$, the algebra $U_z$ is simple.
\end{theorem}

\begin{proof}
Since $SL_2(\mathbb{C})$ acts transitively on $V(I)\setminus\lbrace 0\rbrace,$
we may assume without loss of generality that $a$ is the point with coordinates
$e=1, h=f=0.$
It is straightforward to check that the algebra $U_{z}[e^{-\frac{1}{2}}]$ is
generated by $e^{\frac{1}{2}}, e^{\frac{-1}{2}}$ over $U_z$ subject
to the following relations:
\begin{eqnarray*}
&& [f, e^m] = -mhe^{m-1}+m(m-1)e^{m-1}, \qquad [y,
e^m] = -me^{m-1}xe^{m-1},\\
&& [h, e^m] = 2me^m, \qquad [x, e^m]=0
\end{eqnarray*}

\noindent for $m\in \frac{1}{2}\mathbb{Z}$.
We have that $[\frac{1}{2}e^{-1}h, e-1]=1,$ and both
$\frac{1}{2}e ^{-1}h, e-1$ commute with $A = e^{\frac{-1}{2}}x$,
$C = e^{\frac{1}{2}} y + \frac{1}{2} hA$, $\Delta = h^2 + 4fe + 2h$. Let
us denote the subalgebra of $ U_{z}[e^{-\frac{1}{2}}]$
generated by $A, C, \Delta$ by $U_z'$.
It is easy to see that  $U_{z}[e^{-\frac{1}{2}}] =
\mathbb{C}[\frac{1}{2}e^{-1}h, e-1][e^{-\frac{1}{2}}] \otimes U'_z$.
Direct computation shows that the following relations hold in $U'_z$:
\begin{eqnarray*}
&& [\Delta, C]=\Delta A+(A-C), [A, C]=z'-\frac{1}{2}A^2,\\
&& [\Delta, A]=4C-A, z+\frac{1}{4}\Delta A^2-\frac{1}{2}AC=C^2.
\end{eqnarray*}

Now we will need to recall the explicit relations for $\bggo{}^\lambda$ (noncommutative
deformations of type $D_{n+2}$ singularities).
Recall that Levy \cite{Le} has defined the following algebras
$D(Q,\gamma)$ for a polynomial $Q(t)$, and $\gamma\in \mathbb{C}$, with
generators $u,v,w$, and relations:
$$[u, v]=2w, \qquad [u,w] = −2uv + 2w + \gamma, \qquad [v,w] = v^2 +
P(u)$$ and $$Q(u) + uv^2 + w^2 − 2wv − \gamma v = 0,$$

\noindent where $P$ is the unique polynomial such that
$$Q(-s(s - 1)) - Q(-s(s + 1)) = (s - 1)P(-s(s - 1)) + (s + 1)P(-s(s+
1)).$$

Similarly, Boddington has defined an algebra $D(q)$ depending on a
polynomial $q$, and has showed that $D(q)$ is isomorphic to 
$D(Q, \gamma)$ when $\gamma=−2q(\frac{−1}{2})$ and
$Q(-u+\frac{-1}{4})=[-\sqrt{u}-\frac{1}{2}p(\sqrt{u})],$ where
$p(x)=\frac{-4q(x)q(-x-1)+\gamma^2}{(1+2x)^2}$ and notation
$[f(\sqrt{x})]\in \mathbb{C}[x]$  for a polynomial $f(x)$ means the
following: $f(\sqrt{x})=h(x)+\sqrt{x}[f(\sqrt{x})]$ for unique
polynomials $h(x), [f(\sqrt{x})].$ Boddington \cite{Bo} shows that $D(q)$
is isomorphic to $\bggo{}^{\lambda}$ where $\lambda$ is the tuple
$(\lambda_a, \lambda_b,\lambda_1,\cdots, \lambda_{n-1}, \lambda_c,
\lambda_d)$ such that $q(x)=\prod_{i=0}^{n-1}(x+\mu_i),$ where 
$$\mu_0=\frac{1}{2}\lambda_a-\lambda_b,
\mu_1=\frac{1}{2}(\lambda_a+\lambda_b), \mu_2=\mu_1+\lambda_1,\cdots,
\mu_{n-1}=\mu_1+\lambda_1+\cdots+\lambda_{n-1}+\lambda_c$$

Direct computation shows that our algebra $U'_z$ is isomorphic to $D(Q, 0)$ for
$Q(u)=3z'(-u-\frac{3}{4})-z(-u-\frac{3}{4}).$
This can be seen by putting first $\Delta=-u, C=\frac{1}{2}w, A=-v,$ and
then replacing $u,w$ by $u + \frac{3}{4}$ and $w + \frac{1}{2}$
respectively. (We denote the corresponding parameter by $\lambda(z)$.)
Then we see that for generic $z$, $\lambda(z)$ is the generic parameter
such that $q(\frac{1}{2})=0,$ i.e. $\mu_i=-\frac{1}{2}$ for some $i.$
Thus, $U'_z$ is isomorphic to $\bggo{}^{\lambda(z)},$ and for generic $z,$ $\bggo{}^{\lambda(z)}$ is simple.

To summarize, 
$U_z[e^{-\frac{1}{2}}]\cong W_1(\mathbb{C})[e^{-\frac{1}{2}}]\otimes \bggo{}^{\lambda(z)},$
where the Weyl algebra $W_1(\mathbb{C})$ is defined by the following
generators and relations: $W_1(\mathbb{C})=\mathbb{C}\langle e, \frac{1}{2}e^{-1}h\rangle/([e, \frac{1}{2}e^{-1}h]-1).$ Similarly
we can establish an isomorphism on the level of Rees algebras 
$$R(U_z)[({\hbar}^2e)^{\frac{-1}{2}}]\cong R(W_1(\mathbb{C}))[({\hbar}^2e)^{\frac{-1}{2}}]\otimes_{\mathbb{C}[\hbar]} R(\bggo{}^{\lambda(z)})$$
(recall that $\hbar^2e\in R(U_z)$). 
Since, $R(U_z)_{\bar{a}}$ is complete with respect to $(\hbar^2e-1),$ we have that $({\hbar}^2e)^{\frac{-1}{2}}\in R(U_z)_{\bar{a}}.$ This
and the above isomorphism yields the desired isomorphism 
$R(U_z)_{\bar{a}}\cong W_{\hbar}(Y_{\bar{a}})\otimes_{\mathbb{C}[[\hbar]]} R(\bggo{}_{\lambda(z)})_{\bar{0}}.$

For generic $z$, $U_z$ has no finite-dimensional representations (this was
shown in \cite{KT}). Thus, simplicity of $U_z$ for generic $z$ follows from
Proposition \ref{simple}, since $\bggo{}^{\lambda(z)}$ has no
nontrivial finite-dimensional representations if $\lambda(z)\cdot
\alpha\neq 0$ for all non-Dynkin roots, according to \cite{CBH}.

\end{proof}

\section{Analogue of Bernstein's inequality}
We start by recalling few standard results (whose proofs are
given for the convenience of the reader), which will be used for the proofs
of Theorem \ref{muh}, Theorem \ref{zing}.

The proof of the following Proposition
follows directly a well-known proof of Bernstein's inequality
given by Joseph \cite{GM}.

\begin{prop}\label{Bern}
Let $R$ be a Noetherian ring with finite GK-dimension and
let $M$ be a finitely generated $W_1(R)$-module.
Then for any finite dimensional $\mathbb{C}$-subspace $N\subset M$
\\we have 
$GK_R(RN)\leq GK_{W_1(R)}M+GK_{W_1(\mathbb{C})}(W_1(\mathbb{C})N)-2.$ 
In particular $GK_R(RN)\leq 2(GK_{W_1(R)}(M)-1).$
\end{prop}
\begin{proof}
Let $R_i$ be an ascending filtration of $R$ by finite dimensional
$\mathbb{C}$-spaces, such that $R_0=\mathbb{C}, R_nR_m=R_{n+m}.$
Denote by $A_n$ the $n$-the degree part of the Bernstein filtration
on $W_1(\mathbb{C}),$ so $A_n=\sum_{i+j\leq n} \mathbb{C}x^iy^j.$
Put $U_n=\sum_{i\leq n} A_iR_{n-i}.$ Also put $N^n=U_nN.$
It is easy to check that the kernel of the multiplication map 
$f_n:U_n\to Hom_{\mathbb{C}}(A_nN, N^{2n})$ is 
$\sum A_i(R_{n-i}\cap Ann(N)).$
Thus, $\sum A_i(R_{n-i}/Ann(N))$ injects into $Hom_{\mathbb{C}}(A_nN, N^{2n}).$
So, $$\frac{1}{\dim N}\sum_{i\leq n}(i+1)\dim R_{n-i}N\leq \dim N^{2n}\dim A_nN.$$
Therefore
$$GK_R(RN)\leq GK_{W_1(R)}M+GK_{W_1(\mathbb{C})}(W_1(\mathbb{C})N)-2\leq 2(GK_{W_1(R)}(M)-1).$$

\end{proof}

\begin{cor}\label{ann}
Let $R$ be a Noetherian ring with finite GK-dimension and
let $M$ be a finitely generated $W_1(R)$-module such that $GK_{W_1(R)}(M)=1.$
Then $GK(R/Ann_R(M))=0.$
\end{cor}
\begin{proof}

Let $N\subset M, \dim N< \infty$ be such that $W_1(R)N=M.$
Then by Proposition \ref{Bern}, $GK_R(RN)=0,$ so $GK(R/Ann_R(RN))=GK(R/Ann_R(M))=0.$
\end{proof}

\begin{lemma}\label{square}

Let $R$ be an affine Noetherian $\mathbb{C}$-algebra, and let $e\in R$ be
an element such that $\ad(e)$ is locally nilpotent on $R.$ Then for
any finitely generated $R$-module $M,$ we have  
$GK_{R[e^{-\frac{1}{2}}]}M[e^{-\frac{1}{2}}]\leq GK_RM.$

\end{lemma}

\begin{proof}

Without loss of generality we may assume that $M$ is $e$-torsion free.
Let $V\subset R$ be an $\ad(e)$-stable finite dimensional 
generating subspace of $R$ such that $1, e\in V.$
Let $M_0\subset M$ be a subspace such that $RM_0=M, \dim_{\mathbb{C}}M_0<\infty.$ Since
$e^{-i}(eV^{n-i-1})\subset e^{-(i-1)}V^{n-(i-1)},$ we get that
$$\dim_{\mathbb{C}}(\sum_{i\leq n}e^{-i}V^{n-i}M_0)\leq 
\dim_{\mathbb{C}}\sum_{i\leq n}(V^{n-i}M_0/eV^{n-i-1}M_0).$$
Therefore, $\dim_{\mathbb{C}}\sum_{i\leq n}e^{-i}V^{n-i}M_0=O(n^d)$ where $d=GK_RM.$
There exists $m>0,$ such that 
$Ve^{-1}\subset\sum_{i\leq m}e^{-i}V, Ve^{\frac{1}{2}}\subset e^{\frac{1}{2}}\sum_{i\leq m}e^{-i}V.$ 
Put \\$L=\mathbb{C}e^{-1}+\mathbb{C}e^{\frac{1}{2}}+V.$
Then $L$ generates $R[e^{-\frac{1}{2}}]$ and \\
$L^n\subset \sum_{i\leq mn} (e^{-i}V^{mn-i}+e^{\frac{1}{2}}e^{-i}V^{mn-i}).$
Thus we conclude that $\dim L^nM_0=O(n^d),$ so 
$GK_{R[e^{-\frac{1}{2}}]}M[e^{-\frac{1}{2}}]\leq d.$
\end{proof}

\begin{lemma}\label{quotientss}

Let $M$ be a finitely generated $R$-module which can be filtered with 
with  $R$-submodules $M_0\subset\cdots \subset M_n=M,$ such that 
$GK(M_i/M_{i-1})\geq \frac{1}{2} GK(R/Ann(M_i/M_{i-1}))$ for all $i=1,\cdots ,n.$\\
Then $GK(M)\geq \frac{1}{2} GK(R/\Ann(M)).$
\end{lemma}

\begin{proof}
We will proceed by induction on $i$ to show that \\$GK(M_i)\geq \frac{1}{2} GK(R/\Ann(M_i)).$
Put $I=Ann(M_{i-1}), J=Ann(M_i/M_{i-1})$ then $$GK(H/Ann(M_i))\leq GK(R/IJ)\leq Max \lbrace GK(R/I), GK(R/J)\rbrace$$
and $GK(M_i)\geq Max \lbrace GK(M_{i-1}), GK(M_i/M_{i-1})\rbrace.$ \\Thus 
$GK(M_i)\geq \frac{1}{2}GK(R/Ann(M_i)).$
\end{proof}

\begin{theorem}\label{muh}
 For any $z$ and any
finitely generated $U_z$-module $M$, one has $GK(M)\geq \frac{1}{2}
GK(U_z/\Ann(M)).$
\end{theorem}
 
\begin{proof}

Recall that $V(I)$ is the singular locus of $\spec \Gr U_z=\spec B_{n},$
where $I=(x, y, \Delta)$.
 Let $M$ be a finitely generated $U_z$-module. 
 If $SS(M)$ (the characteristic variety of $M$) 
intersects with $\spec B_n \setminus V(I)$ (which
is the smooth locus of
$\spec B_n$), then
it follows from Gabber's theorem that 
$SS(M)\cap (\spec B_n \setminus V(I))$ is a coisortropic
subvariety of a symplectic variety $\spec B_n \setminus V(I),$ thus
$GK(M)\geq \frac{1}{2}GK(U_z)$. 
 On the other hand, If $SS(M)=\lbrace 0\rbrace,$ then $\dim M$ is finite, so $GK(U_z/\Ann(M))=0$ and we are done.
Therefore, we may assume that $\{0\}\neq SS(M)\subset V(I).$ 

let $a \in SS(M), a\neq 0.$
We will proceed by the induction on the number of irreducible components
of $SS(M).$ If $SS(M)=V(I)$, then $GK(M)=2$ and there is nothing to prove.
Thus we may assume that $SS(M)$ is a union of finitely many lines in $V(I)$
through the origin, equivalently $GK(M)=1.$ We may assume without loss of generality that $a=(e-1, f, h).$
 In particular,
$e\notin \sqrt{\Ann(\Gr M)}.$ Denote by $M'=\lbrace m\in M: e^nm=0, n>>0\rbrace.$  
Thus, $M'$ is a $U_z$-submodule of $M$ and $M"=M/M'$ is $e$-torsion free $U_z$-module.
Then, $SS(M')\subset SS(M)$ and $a\notin SS(M').$ Hence the number
of irreducible components of $SS(M')$ is less
than that of $SS(M).$ Therefore, by the induction
assumption $GK(U_z/\Ann(M'))\leq 2.$ By Lemma \ref{square}, we have that
$GK_{U_z[e^{-\frac{1}{2}}]}(M"[e^{-\frac{1}{2}}])\leq 1.$ 
Let us denote
$\mathbb{C}[e^{-1}h, e-1]\otimes U'_z=W_1(U'_z)$ by $U"_z.$ As was shown
in the proof of Theorem \ref{completion}, $e\in W_1(\bggo{}^{\lambda(z)}), 
U_z[e^{-1/2}]=W_1(\bggo{}^{\lambda(z)})[e^{-1/2}].$
Let $N\subset M"[e^{-\frac{1}{2}}]$ be a finitely generated $W_1(\bggo{}^{\lambda(z)})$-module
which generates $M"[e^{-\frac{1}{2}}]$ over $U_z[e^{-\frac{1}{2}}].$
Thus, $N[e^{-\frac{1}{2}}]=M"[e^{\frac{-1}{2}}].$ Then (by Lemma \ref{square})
$GK_{W_1(\bggo{}^{\lambda(z)})}(N)\leq GK_{U_z[e^{-\frac{1}{2}}]}(M"[e^{-\frac{1}{2}}])\leq 1.$
Therefore, $GK(\bggo{}^{\lambda(z)}/Ann(N))$=0 by Corollary \ref{ann}.
Then, there is a nonzero polynomial $g,$
such that $g(\Delta)N = 0$.
Similarly, there is a nonzero polynomial $\phi,$ such that
$\phi(e^{-1}x^2) \in Ann(N)$ (recall that $\Delta, e^{-1}x^2\in \bggo{}^{\lambda(z)}$).
 Hence 
$e^m \phi(e^{-1}x^2) \in \Ann(M")$ for all $m$. Choosing $m,$ such that $e^m
\phi(e^{-1}x^2) \in U_z$, we conclude that there is a nonzero
$\psi\in \mathbb{C}[e, x],$ such that $\psi\in \Ann(N).$ 
Since $g(\Delta), \psi$ commute with $e^{-\frac{1}{2}}$ and $M"\subset \sum_ie^{-\frac{i}{2}}N$
we conclude that $g(\Delta), \psi\in Ann_{U_z}M".$ So,
$GK(U_z/\Ann(M"))\leq 2.$ Applying Lemma \ref{quotientss} we are done.
\end{proof}
\begin{theorem}\label{zing}
If $M$ is a finitely generated $H_z$-module, then
$GK_{H_z}(M)\geq \frac{1}{2}GK(H_z/Ann(M))$
\end{theorem}

\begin{proof}
Throughout we will suppress index $z$ for simplicity. 
Let $M$ be a finitely generated $\mathbb{C}[t]$-torsion free $H$-module.
Let $\rho: \mathbb{C}(t)\to F$ be an embedding into an algebraically
closed filed $F.$ Put $H_F=H\otimes_{\mathbb{C}}F.$
Then $U_F=H_F/(t-\rho(t))=H(t)\otimes _{\mathbb{C}{t}}F$ and $M_F=M(t)\otimes_{\mathbb{C}(t)}F$
is a finitely generated module over  $U_F.$ So, we may
apply Theorem \ref{muh} to $U_F, M_F$ (where instead of $\mathbb{C}$ the ground field
is $F$). Thus, $2GK_{U_F}(M_F)\geq GK_F(U_F/Ann(M_F)).$ But 
$\Ann_{U_F}M_F=Ann_{H(t)}(M(t)\otimes_{\mathbb{C}(t)}F),$ so \\
$U_F/Ann(M_F)=(H(t)/AnnM(t))\otimes_{\mathbb{C}(t)}F.$ Therefore,
$2GK_{H(t)}(M(t))\geq GK(H(t)/Ann(M).$ On the other hand $GK_H(M)-1=GK_{H(t)}(M(t))$
and $GK_{\mathbb{C}(t)}(H(t)/Ann(M))=GK(H/Ann(M))-1.$ \\Thus,
$GK_H(M)\geq \frac{1}{2}GK(H/Ann(M)).$

Let $M$ be an arbitrary finitely generated $H$-modules. Let $M'\subset M$
be the submodule of all $\mathbb{C}[t]$-torsion modules. 
Thus, $M/M'$  is  $\mathbb{C}[t]$-torsion free and $M'$ can be filtered by 
$H$-submodules, such that
each subquotient is annihilated by some $t-\lambda$ for some
$\lambda\in \mathbb{C}.$
Then $M$ can be filtered
by $H$-submodules, such that for each subquotient the conclusion
of the theorem holds. By
Lemma \ref{quotientss} we are done.

\end{proof}
\section{Equivalence}
In this section we establish an equivalence between the category
of (generalized) Whittaker $U_z$-modules and the category of $\bggo{}^{\lambda(z)}$-modules, an analogue
of Skryabin's equivalence \cite{S}.
This equivalence is a direct consequence of the exactness of the Whittaker
functor for $\mf{sl}_2$ (simplest case of Kostant's result) and the isomorphism
$U_z[e^{\frac{-1}{2}}]\cong W_1(\bggo{}^{\lambda(z)}).$

Let us briefly recall the setup of the quantum hamiltonian reduction
of algebras. Let $A$ be an associative $\mathbb{C}$-algebra, and let
$\mf{n}\subset A$ be a finite dimensional nilpotent Lie subalgebra under the commutator bracket of $A.$
Suppose that $\ad(\mf{n})$ action of
$\mf{n}$ on $A$ is locally nilpotent.
Let us denote  $(A/A\mf{n})^{\mf{n}}$ by $\mathbb{H}(\mf{n}, A).$ 
Then $\mathbb{H}(\mf{n}, A)=End_A(A/A\mf{n})^{op}$
is an algebra. The full subcategory of $A$-modules consisting of
those $A$-modules on which $\mf{n}$-acts locally nilpotently (Whittaker modules)
will be denoted by $(A,\mf{n})$-mod.
One has a functor $Wh$ (Whittaker functor) from $(A,\mf{n})$-mod
to the category of $\mathbb{H}(\mf{n}, A)$- modules  defined as follows $Wh_A(M)=M^{\mf{n}}=Hom_A(A/An, M).$ There
is a functor in the opposite direction $F(N)=A/A\mf{n}\otimes_{\mathbb{H}(\mf{n}, A)}N.$
Under these assumptions we have the following standard

\begin{prop}\label{skryabin}
Let $B\subset A$ be a $\mathbb{C}$-subalgebra containing $\mf{n}.$
If $H^i(\mf{n}, B/B\mf{n})=0$ for all $i>0$, then the functor $Wh_A$
induces an equivalence. Moreover, if $A/A\mf{n}$ is a faithfully flat right
$\mathbb{H}(\mf{n}, A)$-module, then the inverse is
given by the functor $F.$

\end{prop}
\begin{proof}

At first we check that the functor $Wh_B$ is an exact functor.
Of course this will imply the same about $Wh_A.$  It suffices
to check that for any $B$-module $M$ on which $\mf{n}$ acts
locally nilpotently, one has
$H^i(\mf{n}, M)=0$ for all $i>0.$ 
Let $m$ be a nonegative integer. Assume that for any such $M$ one has $H^i(\mf{n}, M)=0$ for all $i>m.$
This is obviously true for $m=\dim \mf{n}.$
We will proceed by descending on $m.$  
Let $C$ be the full subcategory of all $(B, \mf{n})$-mod  whose objects are $M,$
such that $H^m(\mf{n}, M)=0.$ Clearly $Wh$ is exact on $C.$ 
Also, $C$ is closed under taking quotients,  arbitrary direct sums, extensions and contains $B/B\mf{n}.$
Let $N$ be an object in $(B, \mf{n}).$ Let $N'$ be the sum of all submodules of $N$ that belong to $C.$ Then
$N'$ belongs to $C$ and no nontrivial submodule of $N/N'$ belongs to $C.$ This implies that $N/N'=0,$
otherwise it will contain a nonzero quotient of $B/B\mf{n},$ a contradiction.
Thus, $N$ belongs to $C,$ and $Wh_A$ is an exact functor. This implies that
$A/A\mf{n}$ is a projective generator of $(A, \mf{n})$-mod since $Wh_A=Hom_A(A/A\mf{n}, -).$
Therefore, $(A, \mf{n})$-mod is equivalent to $End_A(A/A\mf{n})^{op}= \mathbb{H}(\mf{n}, A)$
with $Wh_A$ being an equivalence.

 Now suppose
that $A/A\mf{n}$ is a flat right $\mathbb{H}(\mf{n}, A)$-module. Then, $F$ is an exact functor, and
so is $Wh_A(F):\mathbb{H}(\mf{n}, A)$-mod$\to \mathbb{H}(\mf{n}, A)$-mod. 
Clearly $Id$ is a subfunctor of $Wh(F),$ moreover
$Wh(F(\mathbb{H}(\mf{n}, A)))=\mathbb{H}(\mf{n}, A).$ Therefore,  since $Wh(F)$ preserves direct
sums, we get that $Wh(F)=Id.$

\end{proof}

Recall that $\bggo{}^{\lambda}$ denoted the noncommutative deformation of
Kleinian singularity of type $D_{n+2}$ with parameter $\lambda.$

\begin{theorem} 
Let $c$ be a nonzero complex number.
One has $\bggo{}^{\lambda(z)}=(U_z/U_z(e-c))^{e-c}.$
The functor $M\to Wh(M)=\lbrace m\in M: (e-c)m=0\rbrace$
defines an equivalence between the category of $U_z$-modules on which
$(e-c)$ acts locally nilpotently and the category of $\bggo{}^{\lambda(z)}$-modules, 
the inverse functor is given by $N\to F(N)=U_z/U_z(e-c)\otimes_{\bggo{}^{\lambda(z)}}N.$

\end{theorem}
\begin{proof}
Without loss of generality assume that $c=1.$
It is well-known and easy to check that $H^1(\mathbb{C}e, \mf{U}\mf{sl}_2/\mf{U}\mf{sl_2}(e-1))=0.$
Thus, according to Proposition \ref{skryabin}, it suffices to check that
$\bggo{}^{\lambda(z)}$ is isomorphic to $(U_z/U_z(e-1))^{e-1}$ and $U_z/U_z(e-1)$
is a free right $(U_z/U_z(e-1))^{e-1}$-module.  

It was shown in the proof of Theorem \ref{completion} that $U_z[e^{-\frac{1}{2}}]=A[e^{-\frac{1}{2}}]$, where
$A=W_1(\bggo{}^{\lambda(z)})[e^{-\frac{1}{2}}].$ Thus, $U_z/U_z(e-1)=A/A(e^{-\frac{1}{2}}-1),$ and\\
$(U_z/U_z(e-1))^e=(A/A(e^{-\frac{1}{2}}-1))^e.$ But 
$A/A(e^{-\frac{1}{2}}-1)=\mathbb{C}[e^{-1}h]\otimes \bggo{}^{\lambda(z)}$\\
and $(A/A(e^{-\frac{1}{2}}-1))^e=\bggo{}^{\lambda(z)}$ and we are done.
\end{proof}

\end{document}